\documentclass[12pt]{amsart}
\usepackage{amsmath,amsfonts,amssymb,color}
\usepackage{a4wide,amsthm}
\usepackage[english]{babel}
\usepackage{textcomp}
\usepackage{hyperref} 
\usepackage{graphicx}
\usepackage{pdfpages}
\renewcommand{\div}{\operatorname{div}}

{\newtheorem{thm}{Theorem}[section]}
{\newtheorem{prop}[thm]{Proposition}}
{\newtheorem{lemme}[thm]{Lemma}}
{}
{\newtheorem{rem}{Remark}[section]}

\begin{document}

\title[$L^p$ estimates for the homogenization of the Stokes problem]{$L^p$ estimates for the homogenization of Stokes problem in a perforated domain}

\author{ Amina Mecherbet \& Matthieu Hillairet} 
\address{Institut Montpelli\'erain Alexander Grothendieck, Universit\'e de Montpellier, Place Eug\`ene Bataillon, 34095 Montpellier Cedex 5 France}
\email{amina.mecherbet@umontpellier.fr, matthieu.hillairet@umontpellier.fr}
\date{\today}

\maketitle

\begin{abstract}
In this paper, we consider the Stokes equations in a perforated domain. When the number of holes increases while their radius tends to $0$, it is proven in \cite{Desvillettes}, under suitable dilution assumptions, that the solution is well-approximated asymptotically by solving a Stokes-Brinkman equation. We provide here quantitative estimates in $L^p$-norms of this convergence.
\end{abstract}
\section{Introduction}
Let $\Omega$ be a  connected smooth bounded domain in $\mathbb{R}^3.$ Given $N\in \mathbb N,$ we  consider $ \left ( B_i^N\right )_{i\in \{1,\cdots,N\} }$ a family of $N$ balls in $\mathbb{R}^3$ such that: 
$$B_i^N:=B\left (x_i^N,\frac{r_i^N}{N} \right ) \subset \Omega, \text{ for all }i \in \{1,\ldots, N\}.$$ 
Defining the perforated set $\mathcal{F}^N$ by  
$$
 \mathcal F^{N} = \Omega  \setminus \bigcup_{i=1}^{N} B_{i}^{N},
$$
we denote $(u^N,\pi^N) \in H^1(\mathcal F^N) \times L^2_0(\mathcal F^N)$ (here the subscript $0$ fixes that $\pi^N$ has mean
$0$ on $\mathcal F^N$)  the unique solution to the Stokes problem:
\begin{equation} \label{eq_stokesN}
\left\{
\begin{array}{rcl}
- \Delta u^N + \nabla \pi^N &=& 0,  \\
{\div} u^N &= & 0 ,
\end{array}
\right.
\quad \text{ on $\mathcal F^{N}$},
\end{equation}
completed with boundary condtions:
\begin{equation} \label{cab_stokesN}
\left\{
\begin{array}{rcll}
u^N(x) &=& V_i^N  , &  \text{on $\partial  B_i^{N}$} , \\
u^N(x) &=& 0 , & \text{on $\partial \Omega$},
\end{array}
\right.
\end{equation}
where $(V_i^{N})_{i=1,\ldots,N} \in (\mathbb R^3)^{N}$ are given.  In \cite{Desvillettes}, the authors show that, if $r_i^N =1$ uniformly, if the holes are sufficiently dilute and the empirical measures associated to the distributions of $(x_i^N,V_i^N)_{i=1,\ldots,N}$ converge to a sufficiently smooth particle distribution function $f(x,v) {\rm d}x {\rm d}v$, 
then the associated sequence of velocity-fields $(u^N)_{N\in\mathbb N}$ converges weakly to the velocity-field $\bar u$ of the unique solution $(\bar{u},\bar{\pi}) \in H^1(\Omega) \times L^2_0(\Omega)$ 
to the Stokes-Brinkman problem:
\begin{equation} \label{eq_brinkman_intro}
\left\{
\begin{array}{rcl}
 - \Delta \bar{u} + \nabla \bar{\pi} &=&  (j-\rho \bar{u}) ,  \\
{\div} \bar{u} &= & 0 ,
\end{array}
\right.
\quad \text{ on $\Omega$},
\end{equation}
completed with boundary condition:
\begin{equation}\label{cab_brinkman_intro}
\bar{u}=0 \: \text{ on $\partial \Omega$.}
\end{equation}
In \eqref{eq_brinkman_intro}, the flux $j$ and density $\rho$ are computed respectively to the given particle distribution function $f$ by:
$$ 
j(x)= 6\pi \int_{\mathbb R^3} v f(x,v) {\rm d}v \qquad \rho(x) = 6\pi \int_{\mathbb R^3} f(x,v) {\rm d}v, \quad \forall  x \in \Omega.
$$ 
We emphasize that here and below (in the definition of discrete densities and fluxes), 
we include the factor $6\pi$ in the formulas. This factor is reminiscent of the Stokes law for the resistance of a viscous fluid on a moving sphere (see next section).
Via  a standard compact-embedding argument, it entails from \cite{Desvillettes} that we have also  strong convergence
of the $u^N$ to $\bar{u}$ in $L^p$-spaces (for $p<6$) up to the extraction of a subsequence. We are interested herein in providing a quantitative estimate 
of the convergence of $u^N$ to $\bar{u}.$ 

\medskip

This problem is related to the homogenization of Stokes problem in perforated domains with homogeneous boundary conditions and a forcing term. In this case, previous studies prove convergence of the sequence of $N$-hole solutions  to the solution of the Stokes-Brinkman problem (or other ones depending on the dilution regime of the holes) in the periodic as in the random setting \cite{Allaire,BeliaevKozlov96,Rubinstein}.
These results extend to the Stokes problem  previous analysis for the Laplace equations \cite{CioranescuMurat}. The problem with non-homogeneous boundary conditions that we consider herein is introduced by \cite{Desvillettes} in a tentative to justify
a Vlasov-Navier-Stokes or Vlasov-Stokes problem that is applied in spray theory \cite{BDMG,Hamdache98}. The strategy here is to couple the Stokes problem \eqref{eq_stokesN}-\eqref{cab_stokesN} by prescribing that the holes are particles
whose position/velocity $(x_i^N,V_i^N)_{i=1,\ldots,N}$ evolve according to Newton laws:
\begin{eqnarray}
\dfrac{\rm d}{{\rm d}t} x_i^N &= & V_i^N,  \label{eq_ode1}\\
m \dfrac{\rm d}{{\rm d}t} V_i^N &=& - \int_{\partial B_i^N} \left( \nabla u + \nabla u^{\top} - p \mathbb I_3 \right) n {\rm d}\sigma. \label{eq_ode2}
 \end{eqnarray}
Here we denote by $m$ the mass of the particles and $n$ the normal to $\partial B_i^N$ directed toward $B_i^N.$ Note that, contrary to the stationary problem we are studying in this paper, in this target system the
holes/particles are moving. As classical in these "many-particle systems", one crucial issue to complete a rigorous derivation
is to control the distance between the particles. Partial improvements have been obtained in this direction either by increasing the family of datas for which transition from the $N$-hole stationary Stokes problem to the Stokes-Brinkman problem hold  \cite{hillairet} or by completing successfully the kinetic program for the  odes \eqref{eq_ode1}-\eqref{eq_ode2} with singular forcing terms \cite{HaurayJabin15}. 
In this paper, we do not tackle this issue on the distance between particles. Keeping in mind that, in the full problem, one wants to couple the dynamical equations for the particles with the pde governing the fluid problem, we infer that a quantitative description 
of the convergence of the $N$-hole solutions to the solutions to the Stokes-Brinkman problem is necessary. So, we discuss in which norms such quantitative estimates may be computed.

\medskip

We make precise now the main assumptions  that are in force throughout the paper:\medskip 
\begin{itemize}
\item the balls are sufficiently spaced:
\begin{equation} \label{hyp1}\tag{H1}
\text{$\exists\, C_0 >0$ independent of $i\neq j,N$ s.t. } {\rm dist}(B_i^{N} ,B_j^{N}) \geq \dfrac{C_0}{N^{\frac 13}},  \quad {\rm dist}(B_i^N,\partial \Omega) \geq \dfrac{C_0}{N^{\frac 13}}; 
\end{equation}

\vskip 8pt

\item the normalized radii $r_i^N>0$ are uniformly bounded: 
\begin{equation}\label{hyp2}\tag{H2}
\exists\, R_0 >0 \text{ independent of $i,N$ s.t. } r_i^N \leq R_0 ;  
\end{equation}

\vskip 8pt

\item the kinetic energies of the data are uniformly bounded: 
\begin{equation}\label{hyp3} \tag{H3}
\exists\,  E_0 >0 \text{ independent of $N$ such that }\dfrac{1}{N} \sum_{i=1}^{N} |V_i^N|^2 \leq |E_0|^2 .  
\end{equation}
\end{itemize}

Then, following \cite{Desvillettes} and \cite{hillairet} we introduce empirical measures to describe the asymptotic behavior of the distribution $(x_i^N, V_i^N, r_i^N)_{i=1,\dots,N} $: 
$$
{S}_N(x,v,r)  = \dfrac{1}{N} \sum_{i=1}^N \delta_{x^N_i,V^N_i,r_i^N}(x,v,r) \in \mathcal{P}(\mathbb{R}^3 \times \mathbb{R}^3 \times ]0,\infty[).
$$
We denote then by $\rho^N$ and $j^N$ its two first momentums: 
$$
\rho^N:= 6\pi \int_{\mathbb R^3 \times ]0,\infty[} S_N({\rm d}v  {\rm d}r),
 \qquad 
 j^N:= 6\pi \int_{\mathbb R^3 \times ]0,\infty[} v S_N({\rm d}v  {\rm d}r). 
$$
The sequence of densities $\rho^N$ (resp. fluxes $j^N$) are then measures (resp. vectorial measures)  on $\mathbb R^3$ with support in $\Omega.$  Compared to \cite{Desvillettes}, the main new assumptions is that the radii of the holes may depend on $i,N.$
We restrict to the dilution regime of this previous reference for simplicity though it is likely that the result extends to the one of \cite{hillairet}.

\medskip

With the above assumptions, for arbitrary $N\in \mathbb N,$ the domain $\mathcal F^N$ has a smooth boundary and there exists a solution to \eqref{eq_stokesN}-\eqref{cab_stokesN}  (see \cite[Section IV]{Galdi}). We have thus at-hand a sequence $(u^N,\pi^N) \in H^1(\mathcal F^N) \times L^2_0(\mathcal F^N)$.
Under assumption \eqref{hyp1}-\eqref{hyp2}-\eqref{hyp3} one may prove that up to the extraction of a subsequence $\rho^N$ (resp. $j^N$) converges to some 
density $\rho \in L^{\infty}(\Omega)$ (resp. flux $j \in L^2(\Omega)$).  We have then a unique solution $(\bar{u},\bar{\pi})$ to the Stokes-Brinkman problem \eqref{eq_brinkman_intro}-\eqref{cab_brinkman_intro} for this density/flux pair (see next section for more details). In order to compute the distance between $u^N$ and $\bar{u}$
we extend $u^N$ to the whole $\Omega$ by setting:
$$
E_\Omega[u^N]:= \Big \{ \begin{array}{lr}
        u^N, & \text{on } \mathcal{F}^N,\\[4pt]
        V_j^N, & \text{on } B_j^N. \\
        \end{array}
$$
Because of boundary conditions \eqref{cab_stokesN}, these extended velocity-fields satisfy $E_{\Omega}[u^N] \in H^1_0(\Omega).$
With these notations, we state now our two results on the convergence of the sequence $(E_{\Omega}[u^N])_{N \in \mathbb N}$ 
towards $\bar{u}.$ 
\begin{thm}\label{theorem1}
Assume that $j \in L^q(\Omega)$ for some  $q>3$ and $p \in ]1,\frac 32[.$ If $R_0 /C_0^3$ is sufficiently small, there exists a constant $K >0$ depending only on $R_0,C_0,p,q,\Omega$ for which:
$$
 \|E_\Omega[u^N]-\bar{u}\|_{L^p(\Omega)} \leq K\left[
\|j^N-j\|_{(\mathcal{C}^{0,1}(\bar{\Omega}))^*}+\| \rho^N -\rho \|_{(\mathcal{C}^{0,1}(\bar{\Omega}))^* } + \frac{\|{j}\|_{L^q(\Omega)} + E_0}{N^{1/3}}\right], 
$$
for $N \geq (4R_0/C_0)^{3/2}.$
\end{thm}
\begin{thm}\label{theorem2}
Given $p \in ]1,\frac 32[$ there exists $K>0$ depending only on $R_0,C_0,p,\|\rho\|_{L^{\infty}(\Omega)},\Omega$ for which:
$$
\|E_\Omega[u^N]-\bar{u}\|_{L^p(\Omega)}\leq K \left[ \|j-j^N\|_{(\mathcal{C}^{0,1}(\bar{\Omega}))^*}+ \left ( \|\rho-\rho^N \|_{(\mathcal{C}^{0,1}(\bar{\Omega}))^*} + \frac{1}{N^{1/3}}  \right )^{1/3}E_0 \right],
$$
for $N \geq (4R_0/C_0)^{3/2}.$
\end{thm}

The two previous theorems give a quantitative estimate of the weak-convergence obtained in \cite{Desvillettes}. They link the convergence of the sequence $(u^N)_{N\in \mathbb N}$ to $\bar{u}$ to the convergence of the fluxes and densities in the so-called bounded-lipschitz or Fortet-Mourier distance (see \cite[Section 6]{Villani}). As the $(\rho^N)_{N\in \mathbb N}$ are positive measures on $\Omega$ with the same finite mass, we may relate the bounded-lipschitz distance $\|\rho-\rho^N \|_{(\mathcal{C}^{0,1}(\bar{\Omega}))^*}$ to the Wasserstein distance between $\rho^N$ and $\rho$ thanks to the Kantorovich-Rubinstein formula \cite[Theorem 5.10]{Villani}.
The restriction on the values $N$ is irrelevant as our aim is to describe the asymptotics $N \to \infty$ of $u^N.$ It is due to the fact that our method requires
that $B(x_j^N,r^N_j/N) \subset B(x_j^N,C_0/4N^{1/3})$ for arbitrary $j\in \{1,\ldots,N\}.$ 

\medskip

The results we state are complementary one to the other.
The first one is limited to sufficiently small ratios $R_0/C_0^3$. This can be interpreted as configurations for which the holes are sufficiently small compared to their relative distances. In this case, the convergence estimate  is linear with respect to the convergence of the data $\rho^N$ and $j^N$.  The second result is valid for
arbitrary data. The counterpart is that the convergence estimate is now sublinear with respect to the convergence of the  densities $\rho^N.$ These results can be
extended in several directions. First, we may interpolate these convergences with crude uniform bounds on $E_{\Omega}[u^N]$ in $L^6(\Omega)$ to extend
the convergence to $L^p$ spaces with $p \geq 3/2$. But we can also generalize the result by considering convergence of the empirical measures in more general 
dual spaces. We comment at the end of the paper on the estimates we can attain with this method.

\medskip

The outline of the paper is as follows. In the next section, we state and prove some technical lemmas on the resolution of the Stokes problem and Stokes-Brinkman
problem. In particular, we state a regularity lemma in negative Sobolev spaces which is at the heart of our computations. Section \ref{sec_proofs} is devoted to the proof
of our main results and we provide a discussion on the possible extensions of our results in a closing section. 

\medskip

We list below some possible non-standard notations that we use during the proofs. First, we use extensively localizing procedures around
the balls $B_j^N$ so that we use repeatedly the shortcut $A(x,r_{int},r_{ext})$ for the annulus with center $x$ and internal (resp. external) radius $r_{int}$ (resp. $r_{ext}$). We also use the notations  $\oint_A u$ for the mean of $u$ on the set of positive measure $A$:
$$
\oint_A u(x){\rm d}x = \dfrac{1}{|A|} \int_{A} u(x){\rm d}x.
$$ 
We denote classically $L^p(\Omega)$ (resp. $W^{m,p}(\Omega)$ or $H^m(\Omega)$) Lebesgue spaces (resp. Sobolev spaces) on $\Omega.$ The index zero 
specifies zero mean when added to Lebesgue spaces and vanishing boundary-values when added to Sobolev spaces. For instance, we denote:
$$
L^2_0(\Omega):= \left\{  v\in L^2(\Omega), \oint_\Omega v =0 \right\} , \quad D_0(\Omega):= \left\{ v \in [H^1_0(\Omega)]^3, \div v = 0 \right\}.
$$
When there is no ambiguity concerning the definition domain, we only use exponents to denote norms: 
$$ \|\cdot \|_q:= \|\cdot\|_{L^q(\Omega)} , \quad \|\cdot \|_{m,q}:= \|\cdot\|_{W^{m,q}(\Omega)}.
$$
Given $\alpha \in (0,1]$ and $\Omega \subset \mathbb R^3,$   we also introduce ${\mathcal C}^{0,\alpha}(\bar{\Omega}),$ the set of $\alpha$-H\"older continuous functions on $\bar{\Omega}.$ When $\Omega$ is bounded, this is a Banach space endowed with the norm:
$$
\|f\|_{\mathcal C^{0,\alpha}(\bar{\Omega})} = \|f\|_{\infty} + \sup_{x \neq y} \dfrac{|f(x) - f(y)|}{|x-y|^{\alpha}}.
$$

\medskip

Given an arbitrary smooth domain $\Omega$ and $q \in (1,\infty),$ we denote $\mathfrak B: L^q_0(\Omega) \to W^{1,q}_0(\Omega)$ the so-called Bogovskii operator (see \cite[Section III.3]{Galdi}). It is a continuous linear map which, given $f \in L^q_0(\Omega)$ provides a solution $u$ to the problem:
$$
\left\{
\begin{array}{rcll}
{\rm div}  u &=& f, & \text{ on $\Omega,$}\\
u& =& 0, & \text{ on $\partial \Omega.$}
\end{array}
\right.
$$
If $\Omega= A(x_0,r_{int},r_{ext}),$ we specify the Bogovskii operator by indices: $\mathfrak B_{x_0,r_{int},r_{ext}}.$ 
Such operators have been extensively studied in \cite{Allaire}. The main results we apply here are summarized in \cite[Appendix A]{hillairet}.

\medskip

Finally, in the whole paper we use the symbol $\lesssim$ to express an inequality with a multiplicative constant depending on irrelevant parameters.

\section{Preliminary results on the  Stokes and Stokes-Brinkman equations}

In this section, we prove some lemmas concerning the resolution of the Stokes and Stokes-Brinkman equations
that will help in the the proofs of our main results. 

\subsection{{Analysis of the Stokes-Brinkman equation in a bounded domain}}


In this whole part $\Omega$ is a fixed smooth bounded domain. Given a boundary condition $u^* \in H^{1/2}(\Omega)$ and 
$\rho \in L^{\infty}(\Omega),$ we consider the Stokes-Brinkman problem: 
\begin{equation} \label{eq_brinkman1}
\left\{
\begin{array}{rcl}
 \rho u - \Delta u + \nabla \pi &=&  j,  \\
{\div} u &= & 0 ,
\end{array}
\right.
\quad \text{ on $\Omega$},
\end{equation}
completed with boundary condition:
\begin{equation}\label{cab_brinkman}
u=u^* \: \text{ on $\Omega$.}
\end{equation}

We assume below that $\rho \geq 0$ including possibly $\rho=0.$ In this latter case, the Stokes-Brinkman equations degenerate into the Stokes equations. 
We refer the reader to \cite[Section IV]{Galdi} for a comprehensive study of Stokes equations. Herein, we also apply the variational characterization of solutions that is provided in \cite[Section 2]{hillairet}. It is straightforward to extend the existence theory of these references to the Stokes-Brinkman equations with an arbitrary bounded weight $\rho \geq0$ yielding the following theorem:




\begin{thm}\label{theorem_existence}
Let $j \in L^{6/5}(\Omega;\mathbb{R}^3)$ and $\rho \in L^\infty(\Omega)$ such that $\rho \geq 0$. Given $u^*\in H^{1/2}(\Omega)$ satisfying:
$$
\int_{\partial \Omega} u^* \cdot n {\rm d}\sigma = 0,
$$
the following equivalent statements hold true and furnish a solution to  \eqref{eq_brinkman1}-\eqref{cab_brinkman}:
\begin{enumerate}
\renewcommand{\labelenumi}{\roman{enumi})}
\item There exists a unique pair $(u,\pi) \in H^1(\Omega)\times L^2_0(\Omega)$ satisfying \eqref{eq_brinkman1} in the sense of $\mathcal D'(\Omega)$
and \eqref{cab_brinkman} in the sense of traces;\\
\item There exists a unique divergence-free $u \in H^1(\Omega)$ satisfying   \eqref{cab_brinkman} in the sense of traces and: 
\begin{equation} \label{ff_brinkman}
\int_{\Omega} \nabla u: \nabla v =  \int_{\Omega} (j-\rho  u ) \cdot v ,\quad \text{ for all $v \in D_0(\Omega)$};
\end{equation}
\item if we assume furthermore that $j=0,$ there exists a unique solution to the minimisation problem:
\begin{equation}\label{formu_varia}
\inf \left \{ \frac{1}{2} \int_\Omega |\nabla v |^2+  \rho |v|^2  ,  v \in [H^1(\Omega)]^3, \div v =0, v = u^* \text{ on } \partial \Omega \right \}.
\end{equation} 
\end{enumerate}
\end{thm}

The proof of this theorem is  a straightforward extension of \cite[Section IV]{Galdi} and \cite[Section 2]{hillairet} and is left to the reader.

\medskip

As stated in  \cite[Theorem IV.6.1]{Galdi}, in the case $\rho=0$ and $u^*=0$ we have also that, if $j \in W^{m,p}(\Omega)$ for some $m \in \mathbb N$ and $p \in (1,\infty)$
then the solution $u$ satisfies $u \in W^{m+2,p}(\Omega).$  We may  extend this regularity statement to our Stokes-Brinkman problem:
\begin{thm}\label{reg_brinkman}
Let $\rho \in L^\infty(\Omega)$ such that $\rho \geq 0$ and assume that $u^*=0,$ $j\in L^q(\Omega)$, for some $q \in [6/5,\infty)$. Then, there exists a unique pair $(u,\pi) \in W^{2,q}(\Omega)\times W^{1,q}(\Omega)$ satisfying \eqref{eq_brinkman1}-\eqref{cab_brinkman}. Moreover, there exists $C=C(\Omega,q,\|\rho\|_{\infty}) >0$ such that: 
$$ \|u\|_{2,q} \leq C \|j\|_{q}.
$$
\end{thm}

\begin{proof}  
Because $\Omega$ is bounded and $q\geq 6/5$ we have that $j\in L^{6/5}(\Omega)$. Theorem \ref{theorem_existence} yields the existence and uniqueness of the solution $(u,\pi) \in H^1(\Omega) \times L^2_0(\Omega)$. 
We recall that we focus on homogenous boundary conditions. In this case $u \in H^1_0(\Omega)$ so that Poincar\'e inequality entails that
$
\|u\|_{1,2} \lesssim \|\nabla u\|_{2}. 
$

\medskip

At first, let assume further that $q \leq 6.$ Because $H^1_0(\Omega) \subset L^q(\Omega),$ we remark that $(u,\pi)$ satisfies the Stokes equation with data $f= j-\rho \, u \in L^q(\Omega)$. 
The regularity theorem for Stokes equations implies that $(u,\pi) \in W^{2,q}(\Omega) \times W^{1,q}(\Omega)$ with:
$$\|u\|_{2,q} \leq C \|j-\rho u \|_{q} \leq C( \|j\|_q+ \|\rho\|_\infty \|u\|_q)
$$
for some positive constant $C>0$ depending only on $\Omega$ and $q.$  
Thus, we want to bound $\|u\|_q$ by $\|j\|_q$. To this end, we apply the weak-formulation of Stokes-Brinkman problem \eqref{ff_brinkman} with $v=u \in D_0(\Omega)$ to get that: 
\begin{align*}
\int_\Omega |\nabla u |^2 &=  \int_\Omega j\cdot u -  \int_\Omega \rho  |u|^2\\
& \leq  \|j\|_q  \|u\|_{q'} \\
& \lesssim  \|j\|_q  \| \nabla u \|_2
\end{align*}
where we applied again the embedding $H^1_0(\Omega) \subset L^{q'}(\Omega)$ since $q \geq 6/5.$ This entails that $\|u\|_{q} \lesssim  \|u \|_{1,2}\lesssim \|j\|_2$ and concludes the proof.

\medskip 

If we assume now $q>6$ we iterate the same argument. Indeed, because $\Omega$ is bounded,  we have in particular that $j \in L^6(\Omega)$ so that the previous reasoning applies yielding:
$$
\|u\|_{2,6} \leq C \|j\|_{L^6(\Omega)} \leq C' \|j\|_{L^q(\Omega)}.
$$
We may then apply the continuous embedding $W^{2,6}(\Omega) \subset W^{1,\infty}(\Omega).$  Hence, we obtain now again that $j- \rho  u \in L^q(\Omega)$ and we conclude by
application of the regularity theorem for Stokes equations as previously. 

\end{proof}



 
Keeping in mind that we want to compare the $N$-solution $E_{\Omega}[u^N]$ with $\bar{u}$ on $\Omega,$ we do not expect to be able to use a regular theory for the Stokes (or Stokes-Brinkman) equations
as above. Indeed, the $u^N$ are solutions to the Stokes equations on $\mathcal F^N$ only. Even if we were extending the pressure $\pi^N$ to $E_{\Omega}[\pi^N]$ by fixing a constant on the $B(x_i^N,r_i^N/N)$ (say $0$ for instance),  we expect that $\Delta E_{\Omega}[u^N] - \nabla E_{\Omega} [\pi^N]$ contains single layer distributions on the interfaces fluid/holes. Fortunately, these single layer distributions are  regular enough to 
compute $L^p$-estimates as depicted below.  These $L^p$-estimates are adapted from weak-regularity statements for stationary Stokes equations that have been obtained in the study of fluid-structure interaction problems \cite[Appendix 1]{Raymond07}.

\medskip

Given $p \in ]1,6[,$ we introduce the following norm of  $v \in H^1_0(\Omega)$: 
$$
[v]_{p,\Omega}:= \sup \left\{ \left|\int_\Omega \nabla v: \nabla w \right|, \; w \in W^{2,p'}(\Omega)\cap W^{1,p'}_0(\Omega), \; \div w=0, \; \|w\|_{2,p'}=1 \right\}.
$$
We have then:
\begin{lemme}\label{lemma_principal}
Let $p\in ]1,6[$. There exists a non-negative $C=C(\Omega,p)$ such that: $$\|v\|_{p} \leq C [v]_{p,\Omega} ,$$
for all divergence-free $v \in H^1_0(\Omega).$
\end{lemme}
Similary, we define the following norm based on the weak-formulation for the Stokes-Brinkman equations. Given $\rho \in L^\infty(\Omega)$ such that $\rho \geq 0,$
we set: 
\begin{multline*}
[v]_{p,\Omega,\rho}:= \sup \Bigg\{ \left|\int_\Omega \nabla v: \nabla w +  \int_\Omega \rho  v \cdot w \right|, \\ w \in W^{2,p'}(\Omega)\cap W^{1,p'}_0(\Omega), \div w=0, \|w\|_{2,p'}=1 \Bigg  \}
\end{multline*}
Then, there holds:
\begin{lemme}\label{lemma_secondaire}
Let  $p\in ]1,6[$ and $\rho \in L^\infty(\Omega)$ such that $\rho \geq 0$.\
 There exists a non-negative $C=C(\Omega,p,\|\rho\|_{\infty})$ that satisfies: $$\|v\|_{p} \leq C [v]_{p,\Omega,\rho},$$
 for all divergence-free $v \in H^1_0(\Omega).$
\end{lemme}
As Lemma \ref{lemma_principal} can be obtained by setting $\rho =0$ in Lemma \ref{lemma_secondaire}, we prove only the second one.
\begin{proof}
The idea is to use the following equality: 
$$\|v\|_{p}= \sup \left\{ \left|\int_{\Omega} v \cdot \phi \right|, \quad \phi \in L^{p'}(\Omega) \quad \|\phi\|_{L^{p'}(\Omega)}=1   \right\}.$$
Let $p \leq 6$ and $\phi \in L^{p'}(\Omega)$, $\|\phi\|_{p'}=1$. Because $p'\geq 6/5,$ we introduce the unique solution $(u_{\phi},\pi_{\phi})$ to the problem
\begin{equation} \label{stokes_phi_2}
\left\{
\begin{array}{rcl}
- \Delta u_\phi + \nabla \pi_\phi+ \rho u_\phi &=& \phi,  \\[4pt]
{\div}\, u_\phi &= & 0 ,
\end{array}
\right.
\quad \text{ on $\Omega$},
\end{equation}
completed with the boundary condition $u_\phi =0$ on $\partial 
\Omega$. According to Lemma \ref{reg_brinkman}, this  solution satisfies  $u_{\phi} \in W^{2,p'}(\Omega) \cap W^{1,p'}_0(\Omega)$, $ p_\phi \in W^{1,p'}(\Omega)$ and 
$$
 \|u_\phi\|_{2,p'} \leq C \|\phi\|_{p'} \leq C.
$$
Moreover, we have that $ W^{2,p'}(\Omega) \subset W^{1,2}(\Omega).$ This yields that, using an integration by parts:
\begin{align*}
 \int_{\Omega} v \cdot \phi  & = \int_{\Omega} v \cdot (- \Delta u_{\phi}+ \nabla \pi_{\phi}+  \rho  u_\phi ) \\
&= \int_{\Omega} \nabla u_{\phi}: \nabla v  + \int_\Omega \rho  u_\phi \cdot v .
\end{align*}
This entails:
$$
 \left| \int_{\Omega} v \cdot \phi \right| \leq [v]_{p,\Omega,\rho}  \|u_\phi\|_{2,p'}  \leq C [v]_{p,\Omega,\rho}.
$$
We obtain:
$$ 
\left|\int_{\Omega} v \cdot \phi \right| \leq C  [v]_{p,\Omega,\rho}, \quad \forall \phi \in L^{p'}(\Omega), \|\phi\|_{p'}=1 .
$$

\end{proof}


\subsection{{The Stokes problem in an exterior domain}}

In this part, we focus on the case $\Omega = \mathbb{R}^3 \setminus B(0,r)$ where $r >0$. 
Given $V \in \mathbb{R}^3,$ we consider the Stokes problem on $\Omega$: 
\begin{equation} \label{eq_stokesunbounded}
\left\{
\begin{array}{rcl}
- \Delta u + \nabla \pi &=& 0,  \\
{\div}\,  u &= & 0 ,
\end{array}
\right.
\quad \text{ on $\Omega$},
\end{equation}
completed with boundary conditions:
\begin{equation} \label{cab_stokesunbounded}
u(x) = V,\:   \text{on $\partial B(0,r),$} \quad \underset{|x|\to \infty}{\lim} |u(x)|=0.
\end{equation}




We investigate here the convergence of Stokes solutions on annuli to the Stokes solution on the exterior domain. Precisely, let $R>r$ and  $\Omega_{R}= B(0,R) \setminus \overline{B(0,r)}=A(0,r,R)$.
We denote by $(u_R,\pi_R)$ the solution to:
\begin{equation}  \label{eq_stokestruncated}
\left\{
\begin{array}{rcl}
- \Delta u_R + \nabla \pi_R &=& 0,  \\
{\div}\, u_R &= & 0 ,
\end{array}
\right.
\quad \text{ on $\Omega_R $},
\end{equation}
completed with boundary conditions:
\begin{equation} \label{cab_stokestruncated}
u(x) = V, \:   \text{on $\partial B(0,r),$} \quad u(x)=0, \:\text{on $\partial B(0,R).$}
\end{equation}
We emphasize that we only consider constant boundary conditions. In this particular case existence theory for \eqref{eq_stokesunbounded}-\eqref{cab_stokesunbounded} is well known since explicit formulas for the solutions
are part of the folklore (see \cite{LL} and  more recently \cite{Desvillettes}). Explicit solutions for \eqref{eq_stokestruncated}-\eqref{cab_stokestruncated} are also available following the same construction scheme as in the unbounded case. We refer here to \cite[Section 6.2]{Desvillettes} for more details. On the basis of these formulas, the convergence of $(u_R,\pi_R)$ to $(u,\pi)$ is studied in  \cite{Desvillettes}. For later purpose, we complement here this study with two supplementary properties of this convergence.

\medskip



First, we denote 
$$
F_R^r = \int_{\partial B(0,r)} (\nabla u_R - \pi_R \mathbb I)  n d \sigma, \quad 
F^r = \int_{\partial B(0,r)} (\nabla u - \pi \mathbb I)  n d \sigma.
$$ 
We use the symbol $\mathbb I$ here for the identity matrix in $\mathbb R^3.$
These quantities are related to the force exerted by the flow $(u_R,\pi_R)$ (resp. $(u,\pi)$) on the hole $B(0,R)$ (see Appendix \ref{force} for more details). 
We recall that Stokes law states that $F^r = 6\pi rV.$
The following lemma shows that the sequence $F_R^r$ converges to $F^r$. Moreover, explicit formulas for $u_R$ and $u$ allow to compute the rate of this convergence:
\begin{lemme}\label{lemma1}
There holds:
$$
|F^r_R-F^r| \lesssim r^2  \frac{|V|}{R}.
$$
\end{lemme}
\begin{proof}
We show the inequality for $r=1$. The result extends to any $r>0$ by a standard scaling argument that we recall afterwards.

\medskip

We have that: 
$$F_R^1-F^1 = \int_{\partial B(0,1)} (\nabla (u_R-u))  n d \sigma + \int_{\partial B(0,1)} (\pi-\pi_R)  n d \sigma .
$$

Adopting the notations introduced in \cite{Desvillettes} we set  $r=|x|$, $\omega= \frac{x}{|x|}$ and $P_\omega V = (\omega \cdot V) \omega$. We have then, for arbitrary $x \in A(0,1,R)$ that:
\begin{align*}
u_R(x) &=-\left [4 A(R)r^2+2B(R)+\frac{C(R)}{r}-\frac{D(R)}{r^3} \right ](\mathbb{I}-P_\omega)V \\
&\; - 2\left [ A(R)r^2+B(R)+\frac{C(R)}{r}+\frac{D(R)}{r^3} \right ]P_\omega V
\end{align*}
where: 
$$
\begin{array}{rclcrcl}
A(R) &=& \displaystyle -\frac{3}{8R^3}+O \left (\frac{1}{R^4} \right ),  & &  B(R) &=& \displaystyle \frac{9}{8R}+O\left(\frac{1}{R^2}\right), \\[12pt] 
C(R) &=& \displaystyle -\frac{3}{4}+O\left (\frac{1}{R}\right ), & &  D(R)&=& \displaystyle \frac{1}{4}+O \left (\frac{1}{R} \right ).
\end{array}
$$
The formula for $u$ is obtained by replacing $A(R),B(R),C(R),D(R)$ by their limits when $R \to \infty$ in the formula
defining $u.$

\medskip

In the same spirit as on page 965 of \cite{Desvillettes}, we have that, for arbitrary $x \in A(0,1,R)$:
\begin{align*}
u_R(x) -  u(x) &\\
=&\left [\frac{3}{2R^3} r^2-r^2 O \left (\frac{1}{R^4} \right )-\frac{9}{4R}+O \left (\frac{1}{R}\right ) +O \left (\frac{1}{R}\right ) \dfrac 1r +\frac{1}{r^3}O\left (\frac{1}{R} \right ) \right ](\mathbb{I}-P_\omega)V \\
  + &\left [\frac{3}{4R^3} r^2+r^2 O \left (\frac{1}{R^4} \right )-\frac{9}{4R}+O\left (\frac{1}{R} \right )+O\left (\frac{1}{R} \right )\frac{1}{r}+\frac{1}{r^3}O \left (\frac{1}{R} \right ) \right ]P_\omega V .
\end{align*}
This yields
\begin{align*}
\int_{\partial B(0,1)}\nabla(u_R-u)  n d \sigma &   \\ 
= & \left(\frac{3}{2R^3}+ O \left( \frac{1}{R^4}\right)+O \left(\frac{1}{R}\right) \right)(4\pi^2 V- \int_{\partial B(0,1)} V\cdot x  x d \sigma )\\ 
+& \left(\frac{3}{2R^3}+O \left (\frac{1}{R^4}\right )+ O \left (\frac{1}{R}\right )\right)\int_{\partial B(0,1)} V\cdot x  x d \sigma,
\end{align*}
and consequently:
$$
\left|\int_{\partial B(0,1)}\nabla(u_R-u)  n d \sigma \right| \lesssim \frac{|V|}{R}.
$$

By using \cite[Section 6.2]{Desvillettes} we get a similary formula for the pressures:
$$
\pi_R(x)-\pi(x)=\left ( -20A(R)|x|+\frac{5A(R)+3B(R)}{|x|^2} \right )
 \frac{x\cdot V}{|x|}, \quad \forall x\in A(0,1,R).
$$

This entails that:

\begin{align*}
\left|\int_{\partial B(0,1)} (\pi-\pi_R)  n{\rm d}\sigma\right| &\lesssim (25 |A(R)|+3|B(R)|) \left| \int_{\partial B(0,1)} V\cdot x  x d\sigma \right|\\
& \lesssim   \frac{|V|}{R}.
\end{align*}
Finally, we get that:
$$
|F_R^1-F^1| \lesssim \frac{|V|}{R} .
$$

We obtain the inequality for arbitrary $r$ by remarking that, denoting $(\tilde{u},\tilde{\pi})$ the solution to the Stokes problem on $\mathbb{R}^3\setminus B(0,r)$ (resp. $(\tilde{u}_R,\tilde{\pi}_R)$ the solution to the Stokes problem on $A(0,r,R)$), we have:
$$
\begin{array}{rcll}
 (\tilde{u}(x),\tilde{\pi}(x))&=& \displaystyle \left (u\left (\frac{x}{r}\right ),\frac{1}{r}\pi \left (\frac{x}{r}\right ) \right ), & \text{ for all }x \in \mathbb{R}^3\setminus B(0,r), \\[12pt]
 (\tilde{u}_R(x),\tilde{\pi}_R(x)&=& \displaystyle \left (u_{R/r}\left (\frac{x}{r} \right ), \frac{1}{r}\pi_{R/r}\left (\frac{x}{r} \right ) \right ), &  \text{ for all }x \in A(0,r,R).
\end{array} 
$$
Introducing this scaling in the formulas for $F_R^r,$  we get that:
$$
F_R^r-F^r= r ( F_{R/r}^1-F^1).
$$
This entails finally that: 
\begin{align*}
|F_R^r-F^r| &= r |F_{R/r}^1-F^1| \\
& \lesssim r^2 \frac{|V|}{R}.
\end{align*}
 
\end{proof}



We conclude this section by an error estimate for the velocity gradient:
\begin{lemme}\label{lemma4}
There holds:
$$ \int_{A(0,R/2,R)} |\nabla u_R|^2 \lesssim  \frac{r^2} {R} |V|^2.
$$
\end{lemme}
\begin{proof}
We obtain the result for $r=1$ by plugging the explicit formulas
for $u^R$ and the coefficients $A(R),B(R),C(R),D(R)$ in the previous proof and generalize it to arbitrary $r >0$ by a scaling argument. The details are left to the reader.
\end{proof}



\section{Proofs of theorem \ref{theorem1} and \ref{theorem2}} \label{sec_proofs}

We proceed in this section with the proofs of our main theorems. In this section, we fix $\Omega,R_0,C_0,$ and $p \in ]1,3/2[,$ $q \in (3,\infty)$ as in the assumptions of our theorems. When using the symbol $\lesssim,$ we allow the implicit constant to depend on theses values $R_0,C_0,p,q,\Omega.$ 

\medskip

Let $N \geq N_0:=  (4R_0/C_0)^{\frac 32}.$ we recall that $E_{\Omega}[u^N]$ is the solution to the Stokes problem \eqref{eq_stokesN}-\eqref{cab_stokesN}  on the perforated domain $\mathcal F^N$ with boundary data $V_1,\ldots,V_N.$ With similar arguments to \cite[Section 3]{hillairet}
we have:
\begin{prop} \label{prop_unifbound}
There exists a constant $K$ depending only on $R_0$ and $C_0$
for which:
$$
\|E_{\Omega}[u^N]\| \leq K E_0, \quad \forall  N \geq N_0.
$$
\end{prop}

We also introduce $\bar{u}$  the solution to the Stokes-Brinkman problem \eqref{eq_brinkman_intro}-\eqref{cab_brinkman_intro}
associated with the data $j,\rho$ that may be computed from the particle distribution function to which the sequence of empirical measures describing the $N$-configurations converges. 

\medskip
  
The main idea is common to both proofs: we apply duality arguments reported in Lemma \ref{lemma_principal} or in Lemma \ref{lemma_secondaire} in order to estimate the $L^p$-norm of the vector-field $v^N:= E_\Omega[u^N]-\bar{u}$. 
Hence, the core of the proof is the computation of 
$$\left| \int_{\Omega} \nabla v^N : \nabla w \right|,$$ 
for an arbitrary divergence-free vector-field 
$ w \in W^{2,p'}(\Omega) \cap W^{1,p'}_0(\Omega)$. 

\medskip

In the two next parts, we prepare these computations by fixing a divergence-free vector-field $w \in W^{2,p'}(\Omega) \cap W^{1,p'}_0(\Omega).$ We compute equivalent formulas for
$$
\int_{\Omega} \nabla v^N: \nabla w ,
$$ 
and provide some bounds that are relevant for both proofs. 

\medskip

We remind the classical embedding that we use repeatedly below: since $p\in [1,\frac{3}{2}[,$ there holds:
$$
W^{2,p'}(\Omega) \hookrightarrow  \mathcal{C}^{0,1}(\bar{\Omega}).
$$

\subsection{\textbf{Extraction of first order terms}}
Let $N\geq N_0$ and  $w \in W^{2,p'}(\Omega) \cap W^{1,p'}_0(\Omega)$ be divergence-free.  We have:
 $$\int_{\Omega} \nabla v^N: \nabla w = \int_{\Omega} \nabla E_{\Omega}[u^N]: \nabla w - \int_{\Omega} \nabla \bar u: \nabla w,$$
where 
\begin{eqnarray*}
\int_{\Omega} \nabla \bar u: \nabla w &=& \int_{\Omega}(j(x)- \rho(x) \bar u(x)) \cdot w(x) {\rm d}x, \\
\int_{\Omega} \nabla E_{\Omega}[u^N]: \nabla w &=& \int_{\mathcal{F}^N} \nabla u^N: \nabla w. 
\end{eqnarray*}
In what follows we use the shortcuts:
\begin{equation}
\label{couronnes}
 A_j^N:= A(x_j^N,C_0/4N^{1/3},C_0/2N^{1/3}), \quad {\bar{u}}_j^N := \oint_{A_j^N} u^N.
\end{equation}
and 
$$
\Omega_j^N = B\left (x_j^N,\frac{C_0}{2N^{\frac{1}{3}}} \right )\setminus B_j^N \quad \forall  j,N.
$$
Because of the definition \eqref{hyp1} of $C_0,$  the sets $\Omega_j^N$ are disjoint and cover a subset of $\Omega$. These sets are also annuli but they play a special role to our proof hence the different name. Because $N \geq N_0,$ the sets $\Omega_{j}^N$ are not empty so that their boundaries are made of two concentric spheres. The internal sphere is $\partial B_j^N$ while we denote below $\partial_e \Omega_j^N$ the external sphere.
 
\medskip

We first decompose the scalar product $ \int_{\mathcal{F}^N} \nabla u^N: \nabla w$ into $N$ integrals on the disjoint annuli $\Omega_j^N$. 
To this end, given $j\in \{1,\ldots,N\},$ we define $(\widehat{w}_j^N,\widehat{\pi}_j^N)$ the unique solution to the Stokes problem \begin{equation} \label{eq_stokes2}
\left\{
\begin{array}{rcl}
-\Delta \widehat{w}_j^N + \nabla \widehat{\pi}_j^N &=& 0 , \\[8pt]
{\rm div}\,  \widehat{w}_j^N &= & 0 ,
\end{array}
\right.
\quad \text{ on $\Omega^{N}_{j}$},
\end{equation}
completed with boundary conditions:
\begin{equation} \label{cab_stokes2}
\left\{
\begin{array}{rcll}
\widehat{w}_j^N(x) &=& w(x) , &  \text{on $\partial  B_j^{N},$}  \\[8pt]
\widehat{w}_j^N(x) &=& 0 , & \text{on $\partial_e \Omega_j^N$}.
\end{array}
\right.
\end{equation}
We still denote $\widehat{w}^N_j$ the trivial extension of $\widehat{w}^N_j$ to $\mathcal{F}^N$ and we set 
$$ 
w^N:=\overset{N}{\underset{j=1}{\sum}} \widehat{w}^N_j.
$$ 
We remark then that $w^N$ satisfies:
$$
\left\{
\begin{array}{ll}
 w^N \in H^1(\mathcal F^N), & \\[4pt]
 {\rm div} \, w^N = 0, & \text{ on $\mathcal F^N$,}\\[4pt]
 w^N = w, & \text{ on $\partial \mathcal F^N.$}
\end{array}
\right.
$$
We have then: 
$$
\int_{\mathcal{F}^N}\nabla u^N: \nabla w = \int_{\mathcal{F}^N} \nabla u^N: \nabla ( w - w^N) + \int_{\mathcal{F}^N}\nabla u^N: \nabla w^N.
$$
Because $u^N$ is the solution to the Stokes problem on $\mathcal{F}^N$ and  $w-w^N \in D_0(\mathcal{F}^N)$, the first term on the right-hand side vanishes:
$$
\int_{\mathcal{F}^N}\nabla u^N: \nabla w =\overset{N}{\underset{j=1}{\sum}} \int_{\mathcal{F}^N} \nabla u^N: \nabla \widehat{w}^N_j:= \overset{N}{\underset{j=1}{\sum}}  I_j^N. 
$$

Let denote now by $(w_j^N,\pi_j^N)$ the unique solution to: 
\begin{equation} \label{eq_stokes3}
\left\{
\begin{array}{rcl}
-\Delta w_j^N + \nabla {\pi_j^N} &=& 0 , \\[8pt]
{\rm div} \, w_j^N &= & 0 ,
\end{array}
\right.
\quad \text{ on $\Omega^{N}_{j},$}
\end{equation}
completed with boundary conditions:
\begin{equation} \label{cab_stokes3}
\left\{
\begin{array}{rcll}
w_j^N(x)= &=& w(x^N_j) , &  \text{on $\partial  B_j^{N},$}  \\[8pt]
w_j^N(x) &=& 0 , & \text{on $\partial_{e}\Omega_{j}^N$}.
\end{array}
\right.
\end{equation}

For arbitrary $j=1,\dots,N,$ we have 
$$
I_j^N =  \int_{\Omega_j^N} \nabla u^N: \nabla ( \widehat{w}^N_j - w^N_j) + \int_{\Omega_j^N}\nabla u^N: \nabla w^N_j. 
$$
and we set 
$$
R1_j^N:= \int_{\Omega_j^N} \nabla u^N: \nabla ( \widehat{w}^N_j - w^N_j).
$$
Because $w_j^N$ is a solution to \eqref{eq_stokes3} and  $ u \in H^1(\Omega_j^N)$ is divergence-free we have also that: 
$$
I_j^N=\int_{ \partial B_j^N} [(\nabla w_j^N - \pi_j^N \mathbb{I})\cdot n] \cdot u^N d\sigma + \int_{\partial_e \Omega_{j}^N} [(\nabla w_j^N - \pi_j^N \mathbb{I}) \cdot n] \cdot u^N d\sigma +R1_j^N.
$$
In the first integral, we note that $u^N = V_j^N$ on $\partial B_j^N.$ We then introduce: $$ F_j^N=\int_{ \partial B_j^N} (\nabla w_j^N - p \mathbb{I})\cdot n d\sigma$$
to rewrite the first term: $$ \int_{ \partial B_j^N} [(\nabla w_j^N - \pi_j^N \mathbb{I})\cdot n] \cdot u^N d\sigma = F_j^N \cdot V_j^N.$$

\medskip

As for the second term, we have (recall \eqref{couronnes} for the definition of $\bar{u}_j^N$): 
$$
\int_{\partial_e \Omega_{j}^N} [(\nabla w_j^N - \pi_j^N \mathbb{I}) \cdot n] \cdot u^N d\sigma   = \int_{\partial_e \Omega_{j}^N} [(\nabla w_j^N - \pi_j^N \mathbb{I}) \cdot n] \cdot {\bar{u}}_j^N d\sigma + R2^N_j,
$$
where 
$$
R2_j^N= \int_{\partial_e \Omega_{j}^N} [(\nabla w_j^N - \pi_j^N \mathbb{I}) \cdot n] \cdot (u^N-{\bar u}_j^N) d\sigma. 
$$

\medskip

At this point, we remark that the Stokes system is the divergence form of the conservation of the normal stresses. This yields that:
$$ 
\int_{ \partial B_j^N} (\nabla w_j^N - \pi_j^N \mathbb{I})\cdot n  d\sigma + \int_{\partial_e \Omega_{j}^N} (\nabla w_j^N - \pi_j^N \mathbb{I}) \cdot n  d\sigma = 0.
$$
Consequently, we obtain that:
$$
\int_{\partial_e \Omega_j^N} [(\nabla w_j^N - \pi_j^N \mathbb{I}) \cdot n] \cdot u^N d\sigma  =  R2_j^N - F_j^N \cdot \bar{u}_j^N .
$$
Eventually, plugging the identities above in $\int_{\mathcal F^N} \nabla u^{N}: \nabla w$ yields that:
\begin{eqnarray} \label{egalite_finale}
\int_{\Omega} \nabla v^N: \nabla w 
& =&   \underset{j=1}{\overset{N}{\sum}}  F_{j}^{N} \cdot V^N_j -  \int_{\Omega} j(x) \cdot w(x) {\rm d}x \\
& &-  \left[  \underset{j=1}{\overset{N}{\sum}}  F_{j}^{N} \cdot \bar{u}^N_{j}
 -  \int_{\Omega} \rho(x) \bar{u}(x) \cdot w(x){\rm d}x \right]   \notag \\[8pt]
& &+   R1^{N} + R2^{N}, \notag 
\end{eqnarray}
where:
$$
R1^{N}:= \sum_{j=1}^{N}R1_{j}^{N},  \quad R2^{N}:= \sum_{j=1}^N R2_{j}^{N}.
$$

\subsection{Estimates applied in both proofs.}
We state and prove here several propositions that are useful in the proof of both theorems. 
\begin{prop}\label{prop2}
There holds:
$$ \Big | \underset{k=1}{\overset{N}{\sum}}F_k^N \cdot v_k^N-  \int_\Omega j(x)\cdot w(x) {\rm d}x\Big| \lesssim \left ( \frac{E_0}{N^{2/3}}+ \|j-j^N\|_{(\mathcal{C}^{0,1}(\bar{\Omega}))^*} \right ) \|w\|_{2,p'}.
$$
\end{prop}
\begin{proof}
We define $(W_j^N,\Pi_j^N)$  by: 
\begin{equation} \label{defW}
(w^N_{j}(x),\pi^N_{j}(x)) = (W^N_{j}(N(x - x^N_{j})),N \Pi^N_{j}(N(x-x^N_j))), \quad \forall  x \in \Omega_{j}^N.
\end{equation}
We note that, substituting in the integral yields:
$$
\int_{\partial  B_{j}^{N}} (\nabla w^N_{j} - \pi^N_{j} \mathbb{I}) \cdot n d \sigma = \frac{1}{N} \int_{\partial  B(0,r_j^N)} (\nabla W^N_{j} - \Pi^N_{j} \mathbb{I}) \cdot n d \sigma, 
$$
and:
 \begin{align*} F_j^N & =  \frac{1}{N} \int_{\partial  B(0,r_j^N)}(\nabla W^N_{j}- \Pi^N_{j} \mathbb{I}) \cdot n d \sigma \\ 
& =  \frac{1}{N} \left(\int_{\partial  B(0,r_j^N)}(\nabla W^N_{j}- \Pi^N_{j} \mathbb{I}) \cdot n d \sigma - 6\pi  r_j^Nw(x_j^N) \right) + \frac{6\pi}{N} r^N_j w(x_j^N).
\end{align*}
We remark then that $(W_j^N,\Pi_j^N)$ is solution to:
\begin{equation} \label{eq_stokesunboundedproof}
\left\{
\begin{array}{rcl}
- \Delta W_j^N + \nabla \Pi_j^N &=& 0,  \\[8pt]
{\div} \, W_j^N &= & 0 ,
\end{array}
\right.
\quad \text{ on $B(0,C_0/N^{2/3}) \setminus B(0,r_j^N)$},
\end{equation}
completed with boundary conditions:
\begin{equation} \label{cab_stokesunboundedW}
W_j^N(x) = w(x_j^N),\:   \text{on $\partial B(0,r_j^N),$} \qquad 
W_j^N(x) = 0, \: \text{ on $\partial B(0,C_0/N^{2/3}),$}
\end{equation}
so that Lemma \ref{lemma1} applies. Assumptions \eqref{hyp2} and \eqref{hyp3} then entail that:
\begin{align*} \left | \underset{k=1}{\overset{N}{\sum}} F_k^N \cdot v_k^N  -  \int_{\Omega} j \cdot w \right | & = \Bigg |  \frac{1}{N} \left( \sum_{j=1}^N\int_{\partial  B(0,r_j^N)}(\nabla W^N_{j}- P^N_{j} \mathbb{I}) \cdot n d \sigma -  6\pi r_j^Nw(x_j^N) \right)\cdot v_k^N  \\ 
& +  \frac{6\pi}{N}\underset{k=1}{\overset{N}{\sum}} r_k^N [w(x_k^N)] \cdot v_k^N -  \int_{\Omega} j \cdot w \Bigg | \\ 
& \lesssim  \frac{1}{N} \underset{k=1}{\overset{N}{\sum}} \frac{|r_k^N|^2}{N^{2/3}}|w(x_k^N)||v_k^N|  +  | \langle j^N - j  ,w  \rangle | \\ 
& \lesssim  \frac{E_0}{N^{2/3}}  \|w\|_{{\infty}} +  \|j^N - j \|_{(\mathcal{C}^{0,1}(\bar{\Omega}))^*} \|w\|_{\mathcal{C}^{0,1}(\bar{\Omega})}.
\end{align*}
We conclude the proof by applying the embedding $W^{2,p'}(\Omega) \subset C^{0,1}(\bar{\Omega}).$ 

\end{proof}

\begin{rem}\label{rem_prop2}
A more general estimate can be proved when $p\in ]3/2,3[$. Indeed, in this case we have the Sobolev embedding  $W^{2,p'}(\Omega)\hookrightarrow \mathcal{C}^{0,\min(1, \alpha_p)}(\bar{\Omega})$ with $\alpha_p:= 2-\frac{3}{p'}=-1+\frac{3}{p} \in (0,1)$. Hence, in the last list of inequality, we may 
bound:
$$ 
| \langle j^N - j  ,w  \rangle | \leq \|j^N-j\|_{(\mathcal C^{0,\alpha_p}(\bar{\Omega})) ^*} \|w\|_{2,p'}.
$$

\end{rem}

We complete the joint part of our main proofs by showing that both $R1^N$ and $R2^N$ vanish when $N \to \infty$. First, we have the following proposition: 
\begin{prop}\label{prop4}
There holds
$$ |R1^N| \lesssim \frac{E_0\|w\|_{2,p'}}{N}.
$$
\end{prop}
\begin{proof}
We remind that: 
 $$
 R1^{N} = \sum_{j=1}^N \int_{\Omega} \nabla E_{\Omega}(u^N): \nabla   (\widehat{w}^N_{j}- w^N_{j}).
 $$
We set  $\tilde{w}^N_{j}$ the difference $\widehat{w}^N_{j}- w^N_{j}$, hence: 

 \begin{eqnarray} \notag
 |R1^{N}| &\leq& \|\nabla E_{\Omega}(u^N)\|_ {L^2(\Omega)} \left ( \sum_{j=1}^N \|\nabla \tilde{w}^N_{j} \|^2_{L^2(\Omega^N_{j})} \right )^{1/2} \\
       & \lesssim & E_0 \left (\sum_{j=1}^N \|\nabla \tilde{w}^N_{j} \|^2_{L^2(\Omega^N_{j})} \right )^{1/2} \label{eq_R1first},
 \end{eqnarray}
because of the bound on $E_\Omega[u^N]$ that we obtained in Proposition \ref{prop_unifbound}. 

\medskip

At this point, we remark that, for $j\in \{1,\dots,N\} $ the $\tilde{w}^N_j$ can be associated with a pressure $\tilde{\pi}^N_j$ (namely $\widehat{\pi}^N_j-\pi^N_j$) to get
the unique solution to the Stokes problem:   
\begin{equation} \label{eq_stokes5}
\left\{
\begin{array}{rcl}
- \Delta  \tilde{w}^N_j + \nabla \tilde{\pi}^N_j &=& 0 , \\[8pt]
{\div}\,\tilde{w}^N_j &= & 0 ,
\end{array}
\right.
\quad \text{ on $\Omega^N_{j}$},
\end{equation}
completed with boundary conditions:
\begin{equation} \label{cab_stokes5}
\left\{
\begin{array}{rcll}
\tilde{w}^N_j(x) &=& w(x) - w(x_j) , &  \text{on $\partial B_j^{N}$} , \\[4pt]
\tilde{w}^N_j&=& 0  & \text{on $\partial_e \Omega_j^N$}.
\end{array}
\right.
\end{equation}
The aim is to bound the $H^1_0(\Omega^N_{j})$-norm of $\tilde{w}^N_j$ by constructing a lifting of boundary conditions \eqref{cab_stokes5} and using the variational characterization of $\tilde{w}^N_j$ solution to \eqref{eq_stokes5}-\eqref{cab_stokes5}.

\medskip 

Let $\chi$ be a truncation function equal to 1 on $B(0,R_0)$ and vanishing outside $B(0, 2R_0)$. We set
$\chi^{N}:= \chi(N (x-x^N_j))$ and we denote $v = v_1  +v_2$ where: 
\begin{eqnarray*}
v_1(x) &=& \chi^N(x) (w(x) - w(x^N_j)), \quad \forall  x \in \Omega^N_{j}, \\
v_2  & = & \mathfrak B_{x_j,R_0/N,2R_0/N}[-{\rm div}(v_1)],
\end{eqnarray*}
with $\mathfrak B$ the bogovskii operator (see \cite[Section III.3]{Galdi}). Because $ \div(v_1) = \nabla \chi^{N} \cdot (w - w(x^N_j)) $  has mean $0$ on $\Omega_j^N,$ the vector-field $v_2$ is well-defined. We may then apply \cite[Appendix A, Lemma 15]{hillairet} to get that: 
\begin{align*} 
\int_{\Omega^N_j} |\nabla v|^2  
&  \lesssim  \int_{A(x_j^N,R_0/N,2R_0/N)} |\nabla w(x)|^2 \\
& + N^2 \int_{A(x_j^N,R_0/N,2R_0/N)} |w(x)-w(x_j^N)|^2 \underset{x \in B(0,2R_0)}{sup} |\nabla \chi (x)|^2 \\
& \lesssim  \frac{1}{N^3} \|w\|^2_{W^{2,p'}(\Omega)} .
\end{align*}
We applied here again the embedding $W^{2,p'}(\Omega) \hookrightarrow \mathcal{C}^{0,1}(\bar{\Omega}) $ for $p' > 3$. Finally, we have
$$
 |R1^N| \lesssim \frac{E_0}{N} \|w\|_{W^{2,p'}(\Omega)}.
$$
\medskip
This ends the proof of our estimate.
\end{proof}

\begin{rem}\label{rem_R1^N}
As in Remark \ref{rem_prop2}, a more general result can be obtained for all $p\in]3/2,3[$. In this case, we have that $W^{2,p'}(\Omega)\hookrightarrow \mathcal{C}^{0,\alpha_p}(\bar{\Omega})$, which provides a more general bound for the error term $R1^N$ of the form $\frac{1}{N^{\alpha_p}}$.
\end{rem}

In order to compute the second error term, we need the following lemma. 
We recall that the annuli $A_j^N$ are defined in \eqref{couronnes}. We keep the
convention that $\partial_e A_j^N$ stands for the external sphere bounding $A_{j}^{N}.$ 
\begin{lemme}
For $j=1,\ldots,N,$ let $v^N_{j} \in H^{1}(A_j^N)$ 
satisfy:
 \begin{itemize}
 \item $\div v^N_{j} = 0 $ on $A_j^N;$
 \item the flux of $v^N_{j}$ through the exterior boundary of $A_j^N$  vanishes:
 $$
 \int_{\partial_e A_{j}^N} v^N_{j} \cdot n {\rm d}\sigma =0;
 $$
 \item the mean of $v^N_{j}$ on $A_j^N$vanishes. 
 \end{itemize}
 Then, there holds: 
$$ \label{eq_casA}
 \left| \sum_{j=1}^{N} \int_{\partial_e \Omega_j^N}(\nabla w^N_{j} - \pi^N_{j}\mathbb I)n\cdot v^N_{j} {\rm d}\sigma \right| \lesssim  
  \left( \sum_{j=1}^{N} \|\nabla v^N_{j}\|_{L^2(A_j^N)}^2\right)^{\frac 12} 
  \left(\sum_{j=1}^N \|\nabla w^N_{j}\|_{L^2(A_j^N)}^2 \right)^{\frac 12} .
 $$
  \end{lemme}

\begin{proof}
We begin by introducing a suitable lifting  of ${v^N_{j}}{|_{\partial_e A_{j}^N}}.$

Namely, we introduce a truncation function $\chi$ such that $\chi$ vanishes on $B(0,C_0/4)$ and is equal to $1$ outside $B(0,C_0/3)$. For $j \in \{1,\ldots,N\},$ we denote
$\chi^N_{j} = \chi(N^{1/3}(x-x_j^N))$ and we set:
$$
\tilde{v}_{j} = \tilde{v}_{j,1} + \tilde{v}_{j,2},
$$
where:
$$
\tilde{v}_{j,1} = \chi^N_{j} v^N_{j},
\qquad
\tilde{v}_{j,2} = \mathfrak B_{x^N_j,C_0/4N^{1/3},C_0/2N^{1/3}}[-\div \tilde{v}_{j,1}].
$$

As, by assumption, we have that $v^N_j$ has flux zero on $\partial_e A_{j}^N$ we obtain that $\div \tilde{v}_{j,1}$ has mean zero on $A_j^N$ 
and $\tilde{v}_{j,2}$ is well-defined. For convenience, we also set:
$$
w^N = \sum_{j=1}^{N} \mathbf{1}_{\Omega^N_{j}} w^N_{j}, \quad 
D\tilde{v} = \sum_{j=1}^{N} \mathbf{1}_{\Omega^N_{j}} \nabla \tilde{v}_{j}.
$$ 

At this point, we note that: 
\begin{itemize}
\item on $\partial B_j^N \subset B(x_j^N,C_0/4N^{1/3})$ we have $\chi^N_j=0$ so that $\tilde{v}_{j,1}=0$. As $\tilde{v}_{j,2}=0$ by construction, we get $\tilde{v}_j=0,$ \\[-8pt]
\item on $\partial_e A_{j}^N = \partial_{e} \Omega_j^N,$ we have $\chi^N_j=1$ so that $\tilde{v}_{j,1}=v_j^N$. As, by construction, $\tilde{v}_{j,2}=0,$ we get $ v_j^N=\tilde{v}_j.$ 
\end{itemize}
These remarks entail that:
\begin{align*}
\sum_{j=1}^{N}  \int_{\partial_e A_j^N}(\nabla w^N_{j} - \pi^N_{j}\mathbb I)n\cdot v^N_{j} {\rm d}\sigma&= \sum_{j=1}^{N}  \int_{\partial \Omega^N_{j}}(\nabla w^N_{j} - \pi^N_{j}\mathbb I)n\cdot \tilde{v}_{j} \, {\rm d}\sigma \\ 
&= \sum_{j=1}^{N} \int_{\Omega^N_{j}} \nabla w^N_{j}: \nabla \tilde{v}_{j},\\ 
&= \int_{\Omega} \mathbf{1}_{Supp(D\tilde{v})}\nabla w^N: D\tilde{v} .
\end{align*}

Consequently, we have:
$$
\left| \sum_{j=1}^{N}  \int_{\partial_e \Omega_j^N}(\nabla w^N_{j} - \pi^N_{j}\mathbb I)n\cdot v_{j}^N {\rm d}\sigma \right|
\leq \left( \int_{\Omega}  \mathbf{1}_{Supp(D\tilde{v})} |\nabla  w^N|^2\right)^{\frac{1}{2}} \left(  \int_{\Omega} |D\tilde{v}|^2\right)^{\frac{1}{2}}.
$$
Due to the fact that the supports of the $\Omega_j^N$ are disjoint and cover the support of $D\tilde{v}$, we have:
$$
\int_{\Omega} \mathbf{1}_{Supp(D\tilde{v})} |\nabla  w^{N}|^2 = \sum_{j=1}^{N} \int_{\Omega^N_{j}}  \mathbf{1}_{Supp(D\tilde{v})}  |\nabla w^N_{j}|^2 
$$
where $Supp(D\tilde{v}) \cap \Omega^N_{j}=A_j^N $ so that
$$\int_{\Omega} \mathbf{1}_{Supp(D\tilde{v})} |\nabla  w^N|^2 = 
\sum_{j=1}^{N} \int_{A_j^N} |\nabla w^N_{j}|^2.
$$
With a similar decomposition, we obtain also:
\begin{align*}
\int_{\Omega}  |D\tilde{v}|^2 &= \sum_{j=1}^{N}\int_{A_j^N} |\nabla \tilde{v}_j|^2, \\
&\leq  2\sum_{j=1}^N \int_{A_j^N} |\nabla \tilde{v}_{j,1}|^2 + |\nabla \tilde{v}_{j,2}|^2. 
\end{align*}

As in the proof of the previous proposition, we compute the terms $\nabla \tilde{v}_{j,1}$, $\nabla \tilde{v}_{j,2}$ and use estimates on Bogovskii operator (see \cite[Appendix A, lemma 15]{hillairet}) to get that there exists a positive constant $K_\chi$ such that: 
$$
\int_{A_j^N} |\nabla \tilde{v}_{j,1}|^2 + |\nabla \tilde{v}_{j,2}|^2
 \\
 \leq K_{\chi} \left( \int_{A_j^N} N^{2/3} |v^N_{j}|^2 +
\int_{A_j^N}  |\nabla v^N_{j}|^2\right).
$$
Finally, we apply the Poincar\'e-Wirtinger inequality in the case of annuli (see \cite[Appendix A, Lemma 13]{hillairet}):
there exists a constant $C>0$ independent of $N$ for which:
$$
\int_{A_j^N} |v^N_{j}|^2  \leq C  \left ( \frac{C_0}{2N^{1/3}} \right )^2
\int_{A_j^N} |\nabla v^N_{j}|^2  .
$$ 
Finally we get that 
$$
 \int_{\Omega} |D\tilde{v}|^2 \lesssim  \sum_{j=1}^N \int_{A_j^N}  |\nabla v^N_{j}|^2.
$$
\end{proof}

We may now state the result on the control of the second error term $R_2^N:$
\begin{prop}\label{prop42}
There holds
$$
|R2^N| \lesssim \frac{E_0\|w\|_{2,p'}}{N^{1/3}}.
$$
\end{prop}
 
\begin{proof}
The main idea to compute $R2^N$ is to apply the previous lemma to
$$
v^N_{j} = \mathbf{1}_{A_j^n} \left[ u^N - \oint_{A(x^N_j,C_0/4N^{1/4},C_0/2N^{1/3})} u^N \right].
$$
This entails that 
\begin{equation}\label{R2_fin}
|R2^N| \leq  K   \|\nabla u^N\|_{L^2(\mathcal F^N)} \left(\sum_{i=1}^N \|\nabla w^N_{j}\|^2_{L^2(A_j^N)}\right)^{\frac 12}.
\end{equation}
At this point, we recall the definition of $W_j^N$ (see \eqref{defW}) and use the change of variable $y=N(x-x_j^N)$:
\begin{align*}
\| \nabla w_j^N \|_{L^2(A_j^N)}^2 & = \int_{A_j^N} N^2 | \nabla W_j^N(N(x-x_j^N))|^2 dx \\ 
& =  \frac{1}{N} \int_{A(0,\frac{C_0N^{2/3}}{4},\frac{C_0N^{2/3}}{2}))} | \nabla W_j^N(y)|^2 dy \\ 
& =  \frac{1}{N}\| \nabla W_j^N \|_{L^2(A(0,\frac{C_0N^{2/3}}{4},\frac{C_0N^{2/3}}{2}))}^2.
\end{align*}

\medskip

We may then apply Lemma \ref{lemma4} to get that:
$$\| \nabla W^N_j\|_{L^2(A(0,\frac{C_0N^{2/3}}{4},\frac{C_0N^{2/3}}{2}))}^2 \lesssim |{r_j^N}|^2 \frac{|w(x_j^N)|^2}{C_0N^{2/3}}.$$
Plugging these identities into \eqref{R2_fin}, applying the fact that $E_\Omega(u^N)$ is bounded for the ${D}_0(\Omega)$-norm and assumption \eqref{hyp2}, we obtain: 
$$ |R2^N| \lesssim E_0 \left ( \frac{1}{N} \underset{j=1}{\overset{N}{\sum}}|{r_j^N}|^2 \frac{|w(x_j^N)|^2}{C_0N^{2/3}} \right )^{1/2} \lesssim \frac{E_0}{N^{1/3}}\|w\|_{\infty}.$$
\end{proof}
\subsection{Proof of Theorem \ref{theorem1}}
We now turn to the proof of the theorem including a smallness assumption on the size of the holes.
For this proof, we first complement the computations in the previous section by estimating the term
on the second line of \eqref{egalite_finale}:
\begin{prop}\label{prop3}
Under the further assumption that $j\in L^q(\Omega)$ for some $q >3,$ there exists $K_{p,\Omega}$ depending only on $p$ and $\Omega$ such that:
\begin{multline*}
\Big | \underset{k=1}{\overset{N}{\sum}}F_k^N \cdot \bar{u}_k^N -  \int_\Omega \rho(x) w(x) \cdot \bar{u}(x) \Big | - K_{p,\Omega} \dfrac{R_0}{C_0^3} \|w\|_{2,p'} \|u^N-\bar{u}\|_{p}  \\ \lesssim  \left[ \dfrac{E_0}{N^{2/3}} + \Big (  \| \rho^N - \rho \|_{(\mathcal{C}^{0,1}(\bar{\Omega}))^*}  + \frac{1}{N^{1/3}} 
 \Big ) \|\bar{u}\|_{2,3}\right] \|w\|_{2,p'}
\end{multline*}
\end{prop}
\begin{proof}
We may write:
$$\underset{k=1}{\overset{N}{\sum}}F_k^N \cdot \bar{u}_k^N = \underset{k=1}{\overset{N}{\sum}} \Big(F_k^N -\frac{6\pi}{N} r_k^N w(x_k^N)\Big) \cdot \bar{u}_k^N + \frac{6 \pi}{N} r_k^N w(x_k^N) \cdot \bar{u}_k^N$$
We remind that given $k \in \{1,\dots,N\}$:
$$ |A_k^N|=|B(x_k^N,C_0/2N^{1/3})|-|B(x_k^N,C_0/4N^{1/3})|=\frac{4}{3} \pi \left ( \frac{C_0^3}{8N}-\frac{C_0^3}{64N} \right ) = \frac{7 C_0^3 \pi}{48 N}. $$ 
According to the same computations as in the proof of Proposition \ref{prop2}:
\begin{align*}
\left | \underset{k=1}{\overset{N}{\sum}} \Big(F_k^N -\frac{6\pi}{N} r_k^Nw(x_k^N)\Big) \cdot \bar{u}_k^N \right | & \lesssim \frac{1}{N^{2/3}}
 \underset{k=1}{\overset{N}{\sum}}\frac{1}{N}|\bar{u}_k^N| \|w\|_{\infty} \\
& \lesssim  \frac{1}{N^{2/3}} \|w\|_{\infty}  \frac{1}{N} \underset{k=1}{\overset{N}{\sum}} \frac{1}{|A_k^N|}\int_{A_k^N} |u^N|\\ 
& \lesssim \frac{E_0}{N^{2/3}} \|w\|_{\infty}.
\end{align*}
In order to compute the remaining term we introduce the linear mapping:

\begin{displaymath}
\Pi_N:
\left\{
  \begin{array}{rcl}
   \mathcal{C}^\infty_c(\bar{\Omega}) & \longrightarrow &\mathbb{R} \\
    \phi & \longmapsto &\displaystyle \langle \Pi_N, \phi \rangle:= \frac{6 \pi }{N}\underset{k=1}{\overset{N}{\sum}} r_k^N  w(x_k^N) \cdot \oint_{A_k^N }\phi. \\
  \end{array}
\right.
\end{displaymath}
We also set:
$$ \langle \Pi, \phi \rangle:= \int_{\Omega} \rho(x)w(x) \cdot \phi(x){\rm d}x,$$
to rewrite the term: 
\begin{align*} \frac{1}{N} \underset{k=1}{\overset{N}{\sum}} 6 \pi r_k^N  w(x_k^N) \cdot \bar{u}_k^N - \int_{\Omega} \rho(x) w(x) \cdot \bar{u}(x) {\rm d}x & =  \langle \Pi_N,u^N \rangle - \langle \Pi,\bar{u} \rangle \\
& =  \langle \Pi_N,u^N-\bar{u} \rangle + \langle \Pi_N-\Pi,\bar{u} \rangle.
\end{align*}

By straightforward computations, we show that $(\Pi_N)_N$ is a bounded family of linear mappings on $L^p(\Omega)$.
Indeed, recalling the definition of  $R_0$ in \eqref{hyp2} and the above computation of $|A_k^N|,$ we obtain: 
\begin{align*} | \langle \Pi_N, \phi \rangle | 
& \leq  \|w\|_{\infty} \frac{6\pi}{N|A_1^N|} \underset{k}{\max}r_k^N  \int_{ \underset{k}{\bigsqcup} A_k^N} |\phi| \\
& \leq  K_{p,\Omega}  \dfrac{R_0}{C_0^3} \|w\|_{\infty} \|\phi\|_{p},
\end{align*}
with $K_{p,\Omega}$ depending only on $\Omega$ and $p.$ 
Hence, applying the embedding $W^{2,p'}(\Omega) \subset L^{\infty}(\Omega)$ (with a constant depending only on $p,\Omega$) we obtain with a possibly different constant $K_{p,\Omega},$ keeping the same dependencies:
$$
 | \langle \Pi_N,u^N-\bar{u} \rangle|  \leq K_{p,\Omega} \dfrac{R_0}{C_0^3}\|w\|_{2,p'}\|u^N-\bar{u}\|_{p}.
$$
To compute the last term we use the regularity of $\bar{u}$ solution to the Brinkman problem. Indeed, if $j\in L^q$ for some $q>3$ then Theorem \ref{reg_brinkman} shows that
$\bar{u} \in  W^{2,q}(\Omega) \hookrightarrow \mathcal{C}^{0,1}(\bar{\Omega})$, thus, there holds: 
\begin{align*}
| \langle \Pi_N - \Pi, \bar{u} \rangle | & = 
\Big | \frac{6\pi}{N} \underset{k=1}{\overset{N}{\sum}} r_k^N  w(x_k^N)\cdot \bar{u}(x_k^N) - \int_{\Omega} \rho(x) w(x) \cdot \bar{u}(x){\rm d}x \\
& +  \frac{6\pi}{N}\underset{k=1}{\overset{N}{\sum}} r_k^Nw(x_k^N)\cdot \oint_{A_k^N} (\bar{u}-\bar{u}(x_k^N))  \Big | \\ \\
& \lesssim  | \langle \rho^N - \rho  ,  w \cdot \bar{u} \rangle | \\
& +   \frac{6\pi}{N}\underset{k=1}{\overset{N}{\sum}} r_k^N  |w(x_k^N)|  \frac{C_0}{2N^{1/3}} \| \bar{u}\|_{\mathcal{C}^{0,1}(\bar{\Omega})} \\ \\ 
& \lesssim  \| \rho^N - \rho \|_{(\mathcal{C}^{0,1}(\bar{\Omega}))^*} \|w\|_{\mathcal{C}^{0,1}(\bar{\Omega})} \|\bar{u}\|_{\mathcal{C}^{0,1}(\bar{\Omega})} \\ 
& + \frac{1}{N^{1/3}} \|w\|_{\infty}\|\bar{u}\|_{\mathcal{C}^{0,1}(\bar{\Omega})}.
\end{align*}

\end{proof}
\begin{rem}\label{rem_prop4}
When $j\in L^q(\Omega)$ with $q\in ]3/2,3[$, a similar estimate holds involving the distance between $\rho$ and $\rho^N$ in the dual of $\mathcal{C}^{0,\alpha_q}(\bar{\Omega})$. This restriction is due to the embedding  $\bar{u}\in W^{2,2}(\Omega) \hookrightarrow \mathcal{C}^{0,\alpha_q}(\bar{\Omega})$. The case $q\in ]3/2,3[$ involves also  a remainder term that converges to zero like $\frac 1{N^{\alpha_q/3}}$.\\
\end{rem}

To complete the proof of Theorem \ref{theorem1}, we remind that we introduced an exponent $p \in ]1,3/2[,$ and an arbitrary divergence-free test-function $w \in W^{1,p'}_0(\Omega) \cap W^{2,p'}(\Omega);$ inspired by Lemma \ref{lemma_principal}, we computed  \eqref{egalite_finale} which we recall here: 
\begin{multline} \label{eq_egalitefinale2}
\int_\Omega \nabla (E_{\Omega}[u^N] - \bar{u}):\nabla w =  \left(\underset{j=1}{\overset{N}{\sum}}  F_{j}^{N} \cdot v^N_j -  \int_\Omega j(x) \cdot w(x) dx \right)\\
  +\left(\int_\Omega \rho(x) \bar{u}(x)\cdot w(x)dx - \sum_{j=1}^N F_j^N \cdot \bar{u}^N_{j}\right)
 + R1^{N} + R2^{N}.
\end{multline}

At this point, we apply now propositions \ref{prop2}, \ref{prop4}, \ref{prop42} and \ref{prop3}. This entails that there exists a  constant $K$ depending only on $p,\Omega,R_0,C_0$
for which 
\begin{multline*}
\left| \int_\Omega \nabla (E_{\Omega}[u^N] - \bar{u}):\nabla w\right|
  \leq\|w\|_{2,p'} \\ \Bigg( K   \left[ \|j^N-j\|_{(\mathcal{C}^{0,1}(\bar{\Omega}))^*}  + \frac{E_0+\|\bar{u}\|_{2,q}}{N^{1/3}} + \| \rho^N -\rho \|_{(\mathcal{C}^{0,1}(\Omega))^* } \right]  
+  K_{p,\Omega} \dfrac{R_0}{C_0^3}\|u^N-\bar{u}\|_{p}  \Bigg).
\end{multline*}
Consequently, applying Lemma \ref{lemma_principal} and regularity theory for Stokes-Brinkman problem, we obtain a constant $K_{p,\Omega}$ which may  differ from the previous ones, but still depending only on $p$ and $\Omega,$
such that:
$$
\left(1-K_{p,\Omega} \dfrac{R_0}{C_0^3} \right)\|u^N-\bar{u}\|_{p} \lesssim  \|j^N-j\|_{(\mathcal{C}^{0,1}(\bar{\Omega}))^*}+\| \rho^N -\rho \|_{(\mathcal{C}^{0,1}(\Omega))^* }+\frac{E_0+\|j\|_{q}}{N^{1/3}}.
$$
This yields the expected result assuming that $R_0/C_0^3$ is sufficiently small.

\subsection{Proof of Theorem \ref{theorem2}}
We proceed with the proof of our second main result. We do not consider in this case any particular restriction on the ratio $R_0/C_0^3.$  We want to apply now Lemma \ref{lemma_secondaire}.
So, we remind that for a fixed divergence-free test-function $w \in W^{1,p'}_0(\Omega) \cap W^{2,p'}(\Omega),$ by using again formula \eqref{egalite_finale}, we get:
\begin{multline} \label{eq_brinkmanfinale}
\int_\Omega \nabla [E_{\Omega}[u^N] - \bar{u}]: \nabla w +\int_\Omega \rho [E_{\Omega}[u^N] - \bar{u}]\cdot w  = 
\left(\underset{j=1}{\overset{N}{\sum}}  F_{j}^{N} \cdot v^N_j -  \int_\Omega j\cdot w\right) \\ 
 +\left( \int_\Omega \rho  E_{\Omega}[u^N] \cdot w - \underset{j=1}{\overset{N}{\sum}}  F_{j}^{N} \cdot \bar{u}^N_j\right) 
 +  R1^{N} + R2^{N} .
\end{multline}
In order to treat the new term 
$$ \int_\Omega \rho E_{\Omega}[u^N] \cdot w - \underset{j=1}{\overset{N}{\sum}}  F_{j}^{N} \cdot \bar{u}^N_j,
$$ 
we apply the following proposition:
\begin{prop}\label{prop5}
There holds:
$$
\left |  \int_\Omega \rho  E_{\Omega}[u^N] \cdot w - \underset{j=1}{\overset{N}{\sum}}  F_{j}^{N} \cdot \bar{u}^N_j \right | \lesssim \left ( \|\rho-\rho^N\|_{(\mathcal{C}^{0,1}(\bar{\Omega}))^*}+\frac{1}{N^{1/3}} \right )^{1/3}E_0 \|w\|_{W^{2,p'}(\Omega)}.
$$
\end{prop}
\begin{proof}
Using the same notations $\Pi_N$ and $\Pi$ as in the previous proof , we write:
$$
 \int_\Omega \rho  E_{\Omega}[u^N] \cdot w - \underset{j=1}{\overset{N}{\sum}}  F_{j}^{N} \cdot \bar{u}^N_j =
\underset{k=1}{\overset{N}{\sum}} \left(\frac{6\pi}{N} r_k^N w(x_k^N)-F_k^N\right) \cdot \bar{u}_k^N 
+\langle \Pi - \Pi_N,E_{\Omega}[u^N] \rangle,
$$
where the first quantity on the right-hand side is treated as in Proposition \ref{prop2}:
$$\left | \underset{k=1}{\overset{N}{\sum}} (F_k^N -\frac{6\pi}{N} r_k^Nw(x_k^N)) \cdot \bar{u}_k^N \right | \lesssim \frac{E_0}{N^{2/3}} \|w\|_{\infty}.$$
To compute the second term $| \langle \Pi_N-\Pi,u^N \rangle |$, we remark that for arbitrary smooth test function $\phi$ there holds:
\begin{align*}
\langle \Pi_N-\Pi, \phi \rangle  & = \frac{6\pi}{N} \underset{j=1}{\overset{N}{\sum}} r_j^N w(x_j^N) \cdot \oint_{A_j^N} \phi(x) dx -  \langle \rho, \phi \cdot  w \rangle  \\
&= \langle \rho^N-\rho, \phi \cdot w \rangle + \frac{6\pi}{N} \underset{j=1}{\overset{N}{\sum}} r_j^N w(x_j^N) \oint_{A_j^N} (\phi(x)-\phi(x_j^N))dx, 
\end{align*} 
and consequently,
\begin{align*}
|\langle \Pi_N-\Pi, \phi \rangle| & \lesssim |\langle \rho^N-\rho, \phi \cdot w \rangle |+\dfrac{ 1} {N^{1/3}}  \|w\|_{\infty}   \|\phi\|_{\mathcal{C}^{0,1}(\bar{\Omega})} \\
& \lesssim \left ( \|\rho-\rho^N\|_{(\mathcal{C}^{0,1}(\bar{\Omega}))^*} + \frac{1}{N^{1/3}} \right )  \|\phi\|_{\mathcal{C}^{0,1}(\bar{\Omega})}  \|w\|_{\mathcal{C}^{0,1}(\bar{\Omega})}.
\end{align*}
On the other hand for all $\phi \in L^2(\Omega),$ we have that:
\begin{align*}
| \langle \Pi_N-\Pi, \phi \rangle | & =\Big | \frac{6\pi}{N}  \int_\Omega \underset{j=1}{\overset{N}{\sum}}  \frac{1_{A_j^N}(x)}{|A_j^N|}w(x_j^N) \cdot  \phi(x) dx - \int_{\Omega}\rho \phi \cdot w \Big |\\
& \lesssim  \|w\|_{\infty}  \underset{j=1}{\overset{N}{\sum}}\int_{A_j^N} |\phi| + \|w\|_{\infty} \|\rho\|_{2} \|\phi\|_{2}\\
& \lesssim \|w\|_{\infty}  \|\phi\|_{2}.
\end{align*}
We now propose to interpolate the results above as we want to apply the previous inequalities with $\phi = E_{\Omega}[u^N] \in H^1_0(\Omega)$. So, let $\chi$ be a mollifier having support in $B(0,1)$. We construct then the approximation of unity$$\chi_\delta(\cdot)=\frac{1}{\delta^3}\chi\left (\frac{\cdot}{\delta}\right ), \quad 
\forall \, \delta >0.$$

\medskip

Thanks to the previous computations, we have that:
\begin{align*}
 \left | \langle \Pi_N-\Pi, E_{\Omega}[u^N] \rangle \right | & \leq \left |\langle \Pi_N-\Pi, E_{\Omega}[u^N]*\chi_\delta \rangle \right | + \left | \langle \Pi_N-\Pi, E_{\Omega}[u^N]-E_{\Omega}[u^N]*\chi_\delta \rangle \right | \\
& \lesssim  \left ( \|\rho-\rho^N\|_{(\mathcal{C}^{0,1}(\bar{\Omega}))^*} + \frac{1}{N^{1/3}} \right )  \|E_{\Omega}[u^N]*\chi_\delta \|_{\mathcal{C}^{0,1}(\bar{\Omega})}  \|w\|_{\mathcal{C}^{0,1}(\bar{\Omega})} \\
& + \|w\|_{\infty}  \|E_{\Omega}[u^N]-E_{\Omega}[u^N]*\chi_\delta\|_{2}.
\end{align*}
At this point we remark that $E_{\Omega}[u^N]*\chi_\delta \in H^3(\Omega) \hookrightarrow \mathcal{C}^{0,1}(\bar{\Omega})$ with continuous embedding. Furthermore, straightforward computations show that:
\begin{eqnarray*} 
\|E_{\Omega}[u^N]-E_{\Omega}[u^N]*\chi_\delta \|_{L^2(\mathbb{R}^3)} &\lesssim& \delta  \|u\|_{H^1_0(\Omega)},
\\
\|E_{\Omega}[u^N]*\chi_\delta\|_{H^3(\mathbb{R}^3)} &\lesssim& \frac{1}{\delta^2} \|u\|_{H^1_0(\Omega)}.
\end{eqnarray*}
Plugging these estimates in the previous inequality yields that:
\begin{multline*}
|\langle \Pi_N-\Pi,E_{\Omega}[u^N] \rangle | \\
\leq \left ( \|\rho-\rho^N\|_{(\mathcal{C}^{0,1}(\bar{\Omega}))^*}+\frac{1}{N^{1/3}} \right ) \frac{1}{\delta^2} \|\nabla E_{\Omega}[u^N]\|_{2}\|w\|_{\mathcal{C}^{0,1}}+\delta \|w\|_{L^\infty} \|\nabla E_{\Omega}[u^N]\|_{L^2(\Omega)}.
\end{multline*}
We may then set $\delta= \left ( \|\rho-\rho^N\|_{(\mathcal{C}^{0,1}(\bar{\Omega}))^*}+\frac{1}{N^{1/3}}\right )^{1/3}$ and again apply that 
$W^{2,p'}(\Omega) \subset \mathcal C^{0,1}(\bar{\Omega})$ with the uniform control on $\|\nabla E_{\Omega}[u^N]\|_{2}$ to get that
$$
|\langle \Pi_N-\Pi,E_{\Omega}[u^N] \rangle | \lesssim \left ( \|\rho-\rho^N\|_{(\mathcal{C}^{0,1}(\bar{\Omega}))^*}+\frac{1}{N^{1/3}} \right )^{1/3} E_0 \|w\|_{W^{2,p'}(\Omega)}.
$$
\end{proof}
Similarly as in the proof of the previous theorem, we complete the proof of Theorem \ref{theorem2} by applying propositions  \ref{prop2}, \ref{prop4}, \ref{prop42} and \ref{prop5} to control the right-hand side of \eqref{eq_brinkmanfinale} and refering to Lemma \ref{lemma_secondaire} to conclude.

\section{Final remarks}
In the main theorems of this paper, we measure the distance $E_{\Omega}[u^N]-\bar{u}$ in $L^p$-spaces with respect to the distances between $(\rho^N,j^N)$ and $(\rho,j)$ in the bounded-Lipschitz norms. 
With the same method, we may prove similar estimates when considering some Zolotarev-like distances of the data:
$$
\|\rho^N - \rho\|_{\mathcal C^{0,\alpha}(\bar{\Omega})} + \|j^N - j\|_{\mathcal C^{0,\alpha}(\bar{\Omega})}, \quad  \alpha \in (0,1).
$$ 

Precisely, reproducing the computations of the paper and introducing the remarks \ref{rem_prop2},\ref{rem_R1^N} and \ref{rem_prop4}, we may prove:
\begin{thm}\label{theorem1_bis}
Let $\alpha \in (0,1)$ and $(p,q) \in (1,3/(1+\alpha)) \times (3/(2-\alpha),\infty).$  Assume that $j \in L^q(\Omega)$ and  $R_0 /C_0^3$ is sufficiently small, there exists a constant $K >0$ depending only on $R_0,C_0,p,q,\Omega$ for which:
$$
 \|E_\Omega[u^N]-\bar{u}\|_{L^p(\Omega)} \leq K\left[
\|j^N-j\|_{(\mathcal{C}^{0,\alpha}(\bar{\Omega}))^*}+\| \rho^N -\rho \|_{(\mathcal{C}^{0,\alpha}(\bar{\Omega}))^* } + \frac{E_0}{N^{\min(1/3,\alpha)}} + \frac{\|{j}\|_{L^q(\Omega)}}{N^{\alpha/3}} \right], 
$$
for $N \geq (4R_0/C_0)^{3/2}.$
\end{thm}
\begin{thm}\label{theorem2_bis}
Let $\alpha \in (0,1)$ and $p \in (1,3/(1+\alpha)).$ There exists a constant $K >0$ depending only on $R_0,C_0,p,\|\rho\|_{L^{\infty}(\Omega)},\Omega$ for which:
$$
\|E_\Omega[u^N]-\bar{u}\|_{L^p(\Omega)}\leq K \left[ \|j-j^N\|_{(\mathcal{C}^{0,\alpha}(\bar{\Omega}))^*}+ \left ( \|\rho-\rho^N \|_{(\mathcal{C}^{0,\alpha}(\bar{\Omega}))^*} + \frac{1}{N^{1/3}}  \right )^{1/3}E_0 \right],
$$
for $N \geq (4R_0/C_0)^{3/2}.$
\end{thm}

\appendix
\section{Fluid/Solid interaction}\label{force}
In this part, we assume that $\Omega=\mathbb{R}^3\setminus B(0,r)$ where $r>0$. Let $V\in \mathbb{R}^3$ be fixed in what follows.
We consider the unique pair $(u,\pi)$ solution to the Stokes problem:
\begin{equation} 
\left\{
\begin{array}{rcl}
- \Delta u + \nabla \pi &=& 0,  \\
{\div} u &= & 0 ,
\end{array}
\right.
\quad \text{ on $\Omega $},
\end{equation}
completed with boundary conditions:
\begin{equation}
u(x) = V  ,   \text{on $\partial B(0,r),$} \quad \underset{|x| \to \infty}{\lim} u(x)=0.
\end{equation}
We denote by $F$ the reaction force applied by the obstacle $B(0,r)$ on the fluid, it is defined as:
\begin{equation}
\label{force1}
F=\int_{\partial \Omega}(\nabla u + \nabla u^{\top} -\pi \mathbb{I}) \cdot n d\sigma. 
\end{equation}
The following lemma provides us an equivalent definition of $F$:
\begin{lemme}\label{fluide/solide}
Let $R_0 \geq r$, there holds:
$$F=\int_{\partial B(0,R_0)} [\nabla u -\pi\mathbb{I}] \cdot n d \sigma.$$
\end{lemme}
\begin{proof}
The aim is to prove that for arbitrary $W \in \mathbb{R}^3$:
$$F\cdot W =\int_{\partial B(0,R_0)} [(\nabla u -\pi\mathbb{I}) \cdot n] \cdot
 W d \sigma $$

Fix a vector-field $ w \in \mathcal{C}^\infty_c(\mathbb{R}^3)$ such that $ \div w=0$, $w=W$ on $B(0,R_0)$, extend $u$ by the value $V$ on $B(0,r)$ and still denote $u$ the extension for simplicity.
After integration by parts we obtain:
\begin{align*}
F \cdot W & = \int_{\mathbb{R}^3\setminus B(0,r)} \nabla u: {\nabla w} +\nabla u: {\nabla w}^{\top} \\
& = \int_{\mathbb{R}^3} \nabla  u: {\nabla w} \\
&= \int_{\mathbb{R}^3 \setminus B(0,R_0)} \nabla  u: {\nabla w}  \\
&=\int_{\partial B(0,R_0)} [(\nabla u - \pi \mathbb{I})\cdot n ] \cdot W{\rm d}\sigma \\
\end{align*}
As $\div u =0$ and $w=W$ on $B(0,r)$. 
\end{proof}
\medskip 


\end{document}